\documentclass[a4paper]{amsart}

\usepackage[english]{babel}
\usepackage[latin1]{inputenc}
\usepackage{verbatim}
\usepackage[arrow, matrix, curve]{xy}
\usepackage{amsmath,amssymb,amsfonts,amsthm,amsopn}
\usepackage{nicefrac}
\usepackage{mathtools}

\title{Heights and totally $p$-adic numbers}

\date{\today}

\author{Lukas Pottmeyer}
\address{Fachbereich Mathematik, Universit\"at Basel, 4051 Basel}
\email{lukas.pottmeyer@unibas.ch}

\dedicatory{F\"ur Opa, der mich das Z\"ahlen lehrte}

\DeclareMathOperator{\Gal}{Gal}

\DeclareMathOperator{\PrePer}{PrePer}
\DeclareMathOperator{\Per}{Per}
\DeclareMathOperator{\Res}{Res}

\newtheorem{Theorem}{Theorem}[section]

\newtheorem{Corollary}[Theorem]{Corollary}
\newtheorem{Conjecture}[Theorem]{Conjecture}
\newtheorem{Proposition}[Theorem]{Proposition}
\newtheorem{Lemma}[Theorem]{Lemma}
\newtheorem{question}[Theorem]{Question}

\theoremstyle{definition}
\newtheorem*{Definition}{Definition}

\theoremstyle{remark}
\newtheorem{Example}[Theorem]{Example}
\newtheorem*{Remark}{Remark}

\makeindex

\begin{document}

\begin{abstract}
We study the behavior of canonical height functions $\widehat{h}_f$, associated to rational maps $f$, on totally $p$-adic fields. In particular, we prove that there is a gap between zero and the next smallest value of $\widehat{h}_f$ on the maximal totally $p$-adic field if the map $f$ has at least one periodic point not contained in this field. As an application we prove that there is no infinite subset $X$ in the compositum of all number fields of degree at most $d$ such that $f(X)=X$ for some non-linear polynomial $f$. This answers a question of W. Narkiewicz from 1963.
\end{abstract}

\maketitle

\section{Introduction}

We fix an algebraic closure $\overline{\mathbb{Q}}$ of the rational numbers, and all algebraic extensions of $\mathbb{Q}$ are assumed to lie in this algebraic closure.
Let $K$ be a number field with non-archimedean valuation $v$, and let $K^{tv}$ be the maximal totally $v$-adic field extension of $K$, that is, the maximal Galois extension of $K$ which can be embedded into $K_v$. Here and in the following, we denote by $K_v$ the completion of $K$ with respect to $v$.

We denote by $h$ the absolute logarithmic Weil height on $\mathbb{P}^1(\overline{\mathbb{Q}})$. For details on height functions we refer the reader to \cite{BG}. Bombieri and Zannier \cite{BZ01} have introduced the notion of the \emph{Bogomolov property} $(B)$ of a field $F\subseteq \overline{\mathbb{Q}}$. We say that such a field $F$ has property $(B)$ if there is a positive constant $c$ such that $h(\alpha)$ is either zero or bounded from below by $c$ for all $\alpha \in \mathbb{P}^1 (F)$. In the same paper Bombieri and Zannier proved that all fields $K^{tv}$ as above have this property.

More examples of fields with property $(B)$ are: the maximal totally real field extension $\mathbb{Q}^{tr}$ of the rationals (proved by Schinzel \cite{Sch73}), any abelian extension of a given number field (proved by Amoroso and Zannier \cite{AZ00}), and $\mathbb{Q}(E_{\rm tors})$ where $E$ is an elliptic curve defined over $\mathbb{Q}$ (proved by Habegger \cite{Ha13}).

In the last three decades the study of height functions associated to a rational function $f\in \overline{\mathbb{Q}}(x)$ of degree at least two has raised increasing interest. We denote such a height function by $\widehat{h}_f$. It is uniquely determined by the properties
\begin{align}\label{canonical height}
 \widehat{h}_f (f(\alpha))=\deg(f)\widehat{h}_f (\alpha)  \hspace{1.0cm} \text{ and } \hspace{1.0cm}  \widehat{h}_f = h + O(1),
 \end{align}
for all $\alpha\in\mathbb{P}^1(\overline{\mathbb{Q}})$ (see \cite{CS93}). We denote with $\Per(f)$ the set of periodic points of $f$ and with $\PrePer(f)$ the set of preperiodic points of $f$. A point is called \emph{preperiodic} if some iterate of this point is periodic. (An equivalent statement is that a point is preperiodic if and only if its forward orbit is finite). It is not hard to see that $\widehat{h}_f$ vanishes precisely on the set $\PrePer(f)$. For a proof and further properties we refer the reader to \cite[Chapter 3]{Si07}. Notice that $h=\widehat{h}_{x^d}$ for any integer $d \geq 2$.

For any rational function $f\in\overline{\mathbb{Q}}(x)$ of degree at least two we can study fields $F\subseteq \overline{\mathbb{Q}}$ with the \emph{Bogomolov property $(B)$ relative to $\widehat{h}_f$}. Here $F$ has property $(B)$ relative to $\widehat{h}_f$ if there is a positive constant $c$ such that $\widehat{h}_f (\alpha) \geq c$ for all $\alpha \in \mathbb{P}^1 (F)\setminus \PrePer(f)$.
Following Fili and Miner \cite{FM12} we say that such a field $F$ has the \emph{strong Bogomolov property $(SB)$ relative to $\widehat{h}_f$} if additionally there are only finitely many preperiodic points of $f$ in $\mathbb{P}^1 (F)$. Note that there are fields with property $(B)$ but without property $(SB)$ relative to $h=\widehat{h}_{x^2}$. For example the maximal abelian field extension $\mathbb{Q}^{ab}$ has property $(B)$ as mentioned above. However, it contains infinitely many roots of unity, which are the preperiodic points of $x^2$. Therefore, $\mathbb{Q}^{ab}$ does not have property $(SB)$.

Apart from the cases where $f$ is a power map, a Chebyshev polynomial or a Latt\`es map there are almost no non-trivial examples of fields $F \subseteq \overline{\mathbb{Q}}$ with property $(B)$ relative to $\widehat{h}_f$. One exception is the field $\mathbb{Q}^{tr}$. In \cite{Po13} we gave a complete classification of rational functions $f \in \overline{\mathbb{Q}}(x)$ such that $\mathbb{Q}^{tr}$ has property $(B)$ relative to $\widehat{h}_f$. This classification is according to whether the Julia set of $f$ is contained in the real line or not. Moreover, it is proven that $\mathbb{Q}^{tr}$ has property $(B)$ relative to $\widehat{h}_f$ if and only if it has property $(SB)$ relative to $\widehat{h}_f$.
In this paper we want to give such a classification result in the non-archimedean setting.

Let again $K$ be a number field with non-archimedean valuation $v$. Then we denote by $\mathbb{C}_v$ the completion of an algebraic closure of $K_v$. It is well known that $v$ extends uniquely to a valuation on $K_v$ and that this valuation extends uniquely to a valuation, which we again denote by $v$, on $\mathbb{C}_v$. The field $\mathbb{C}_v$ is algebraically closed and complete with respect to $v$.
The fields $\overline{\mathbb{Q}}$ and $K_v$ are subfields of $\mathbb{C}_v$, hence we can describe the field $K^{tv}$ as
\[
K^{tv}=\{\alpha \in \overline{\mathbb{Q}} \vert \sigma(\alpha) \in K_v \text{ for all } \sigma \in \Gal(\overline{\mathbb{Q}}/K)\}.
\]

We denote the classical $v$-adic Julia set by $J_v (f)$. Note that $J_v(f)$ might be empty. This is one of the great differences from rational dynamics on $\mathbb{P}^1 (\mathbb{C})$, where the Julia set is always an uncountable set.

The absolute Galois group $\Gal(\overline{\mathbb{Q}}/ \mathbb{Q})$ acts coefficientwise on $\overline{\mathbb{Q}}(x)$. Hence, for any $f\in\overline{\mathbb{Q}}(x)$ and any $\sigma\in\Gal(\overline{\mathbb{Q}}/ \mathbb{Q})$, we have a unique rational function $\sigma(f)$ defined by this action.

The main part of this paper is to prove the following theorem. Although we state it in the classical non-archimedean language, the proof relies heavily on the theory of Berkovich spaces.

\begin{Theorem}\label{totv}
Let $K$ be a number field with a non-archimedean valuation $v$, and let $f\in\overline{\mathbb{Q}}(x)$ be a rational function of degree at least two. The following conditions are equivalent, and each implies the Bogomolov property of $K^{tv}$ relative to $\widehat{h}_f$:

\begin{itemize}
\item[(i)] There exists a $\sigma \in \Gal(\overline{\mathbb{Q}}/K)$ such that $J_v (\sigma(f)) \not\subset \mathbb{P}^{1}(K_v)$ or $J_v (\sigma(f))=\emptyset$.
\item[(ii)] $\vert \PrePer(f) \cap \mathbb{P}^1 (K^{tv}) \vert < \infty$.
\item[(iii)] $ \Per(f) \not\subseteq \mathbb{P}^1 (K^{tv})$. 
\end{itemize}
If $f$ is a polynomial, then {\rm (i)}-{\rm (iii)} are also equivalent to
\begin{itemize}
 \item[(iv)] $\PrePer(f)\not\subseteq \mathbb{P}^1 (K^{tv})$. 
\end{itemize}
\end{Theorem}

We conjecture that $K^{tv}$ having property $(B)$ relative to $\widehat{h}_f$ also implies {\rm (i)} in Theorem \ref{totv}.
In the archimedean setting this has been proved in \cite{Po13} using a classification result for rational functions with Julia set lying on a circle on the Riemann sphere.

Theorem \ref{totv} is a refinement and generalization of the result of Fili and Miner \cite{FM12}. There they prove, using potential theory on the Berkovich line, that {\rm (i)} implies property $(SB)$ of $K^{tv}$ relative to $\widehat{h}_f$, and that the converse is true if $f$ is a polynomial. Our theorem shows that the converse is also true for rational maps, as property $(SB)$ of $K^{tv}$ relative to $\widehat{h}_f$ implies finiteness of $\PrePer(f)\cap \mathbb{P}^1 (K^{tv})$.

In Section \ref{background} we provide some background on non-archimedian dynamical systems.
The proof of Theorem \ref{totv} is presented in Section \ref{proof}. The main ingredient for a proof of the implication ${\rm (i)} \Rightarrow {\rm (ii)}$ is an equidistribution theorem for points of small canonical height, independently achieved by Baker and Rumely \cite{BR06}, Favre and Rivera-Letelier \cite{FRL04}, Chambert-Loir \cite{CL06} and widely generalized by Yuan \cite{Yu08}. The converse implication follows mainly from deep results of Rivera-Letelier \cite{RiL03} on the structure of the Berkovich Julia set of a rational map.

As an application of small heights on totally $p$-adic fields, in Section \ref{Narkiewicz} we investigate an arithmetic field property defined by Narkiewicz, called property $(P)$. A field $F$ is said to have \emph{property $(P)$} if the existence of an infinite set $X\subseteq F$ and a polynomial $f\in F[x]$ such that $f(X)=X$ implies the linearity of $f$.
In 1963 Narkiewicz conjectured that $\mathbb{Q}^{(d)}$, the compositum of all number fields of degree at most $d$, has property $(P)$ for all $d\in \mathbb{N}$. Using a theorem of Bombieri and Zannier, Dvornicich and Zannier gave a positive answer in the case $d=2$ (see \cite{DZ07} or \cite{DZ08}).
We prove in Section \ref{Narkiewicz} that Narkiewicz's conjecture is true in general. %This follows as an application of the mentioned equidistribution theorem.

To connect this conjecture to Theorem \ref{totv}, we sketch how the conjecture follows from our theorem: For every prime $p$ there is a number field $K$ and a valuation $v\mid p$ on $K$ such that the field $\mathbb{Q}^{(d)}$ is contained in $K^{tv}$.
We use this and Theorem \ref{totv} to show that for every polynomial $f \in \overline{\mathbb{Q}}[x]$ of degree at least two, there is a number field $K$ and a non-archimedean valuation $v$ on $K$ such that $\mathbb{Q}^{(d)} \subseteq K^{tv}$ has property $(SB)$ relative to $\widehat{h}_f$.
In particular, this shows that in $\mathbb{Q}^{(d)}$ are neither infinitely many closed finite $f$-orbits nor an infinite $f$-backward orbit of any point. Note that any $f$-backward orbit of a non-preperiodic point leads by \eqref{canonical height} to a sequence of points with canonical height $\widehat{h}_f$ tending to zero. As this is true for all polynomials $f$ of degree at least two, we find that $\mathbb{Q}^{(d)}$ has property $(P)$.

These arguments are valid in a more general setting. We refer to Section \ref{Narkiewicz} for details.

Narkiewicz asked further if any subfield of $\overline{\mathbb{Q}}$ with property $(P)$ is contained in some field $\mathbb{Q}^{(d)}$. This, however, is not true as was shown ineffectively by Kubota and Liardet \cite{KL76}. Effective constructions of counterexamples were presented in \cite{DZ07} and \cite{Wi11}. Our argument leads to another class of counterexamples, which are also presented in Section \ref{Narkiewicz}.

%%%%%%%%%%%%%
%%%%%%%%%%%%%

\section{Background on non-archimedean dynamics}\label{background}

In this section we will present the main definitions from non-archimedean dynamics needed in the rest of this paper. We fix a number field $K$ with non-archimedean valuation $v$, and a rational map $f\in \mathbb{C}_v(x)$ of degree at least two. Throughout the paper we equip $\mathbb{P}^1(\mathbb{C}_v)=\mathbb{C}_v \cup \{\infty\}$ with the topology induced by the $v$-adic chordal metric. Recall that for any $x \in\mathbb{C}_v$, the $v$-adic chordal metric on $\mathbb{P}^1(\mathbb{C}_v)$ is given by
\[
\rho_v ( x, y )=\begin{dcases} \frac{\vert x - y \vert_v}{\max\{\vert x \vert_v ,1\}\max\{\vert y \vert_v ,1\}} & \text{ if } y\neq \infty \\
\frac{1}{\max\{\vert x \vert_v ,1\}} & \text{ if } y = \infty
\end{dcases}.
\]

\begin{Definition}
The \emph{classical $v$-adic Fatou set of $f$}, denoted by $F_v(f)$, is the maximal open subset of $\mathbb{P}^1(\mathbb{C}_v)$ where the iterates of $f$ are equicontinuous at every point of $F_v(f)$. The \emph{classical $v$-adic Julia set of $f$} is defined to be $J_v(f)= \mathbb{P}^1 ( \mathbb{C}_v)\setminus F_v(f)$.
\end{Definition}

Some of the well known properties of the complex Julia set are also preserved by $J_v(f)$. In particular,
\begin{equation}\label{f-invariant}
f(J_v(f))=f^{-1}(J_v(f))=J_v(f).
\end{equation}
We say that $J_v(f)$ is \emph{completely $f$-invariant}. We denote by $R_v =\{\alpha \in \mathbb{C}_v \vert \vert \alpha \vert_v \leq 1\}$ the valuation ring of $\mathbb{C}_v$.

\begin{Definition}
We choose polynomials $g,h \in R_v [x]$ such that:
\begin{itemize}
\item $f=g/h$,
\item $g$ and $h$ have no common zero,
\item at least one coefficient of $g$ or $h$ lies in $R_v^*$. 
\end{itemize}
We say that $f$ has \emph{good reduction} if the resultant $\Res(f,g)$ is in $R_v^*$. A rational map $f\in\overline{\mathbb{Q}}(x)$ has good reduction at $v$ if it has good reduction regarded as a map defined over $\mathbb{C}_v$. Consequently, $f$ has \emph{bad reduction} if it does not have good reduction.
\end{Definition}

Note that this definition does not depend on the particular choice of polynomials $g$ and $h$ satisfying the above criteria. The first statement in the following lemma is obvious from the definition of good reduction. The second statement is also well known and can be found in \cite[Theorem 2.17]{Si07}.

\begin{Lemma}\label{good reduction}
There are only finitely many rational primes $p$ such that $f$ has bad reduction at some place $v\mid p$. If $v$ is a non-archimedean place where $f$ has good reduction, then $J_v(f)$ is empty.
\end{Lemma}

Let $f^{(n)}$ be the $n$-th iterate of $f$, and for any rational map $g$ denote by $g'$ the formal derivative of $g$. A periodic point $\alpha \in \Per(f)\setminus\{\infty\}$ of exact period $n$ is called \emph{repelling} if $\vert (f^{(n)}) ' (\alpha) \vert_v > 1$. A periodic point of $f$ is in $J_v(f)$ if and only if it is repelling \cite[Proposition 5.20]{Si07}.
For more details on classical non-archimedean dynamical systems in dimension one we refer the reader to \cite[Chapters 2 and 5]{Si07}. 

The great disadvantage in studying dynamical systems on $\mathbb{P}^1 (\mathbb{C}_v)$ is that this space is totally disconnected. It turns out that it is more convenient to work in the Berkovich projective line $\mathbb{P}^{1,\mathcal{B}}_{v}$. This is a path-connected Hausdorff-space which contains $\mathbb{P}^1 (\mathbb{C}_v)$ as a dense subspace. We call the topology on $\mathbb{P}^{1,\mathcal{B}}_{v}$ the \emph{Berkovich topology}. Moreover, for any subset $S \subseteq \mathbb{P}^{1}(\mathbb{C}_v)$ we call the closure of $S$, regarded as a subset of $\mathbb{P}^{1,\mathcal{B}}_{v}$, the \emph{Berkovich closure of $S$}. For the general theory of Berkovich spaces we refer the reader to \cite{Ber} and for detailed information on the theory of dynamical systems on the Berkovich projective line we refer to \cite{BR}. 

Every rational function $f \in \overline{\mathbb{Q}}(x)$ of degree at least two leads to a canonical $f$-invariant probability measure $\mu_{f,v}$ on $\mathbb{P}^{1,\mathcal{B}}_{v}$. For a construction of this measure we refer the reader to \cite[\S 10.1]{BR}. The \emph{Berkovich Julia set $J_{v}^{\mathcal{B}}(f)$} of $f$ is defined as the support of $\mu_{f,v}$. Hence, in contrast to the classical $v$-adic Julia set $J_v (f)$, the set $J_{v}^{\mathcal{B}}(f)$ is never empty.
These two notions of non-archimedean Julia sets have the important property 
\begin{equation}\label{Julia}
J_{v}^{\mathcal{B}}(f) \cap \mathbb{P}^1 (\mathbb{C}_v) = J_v (f),
\end{equation} 
which is shown in \cite[Theorem 10.67]{BR}. We will recall a useful characterization of $J_{v}^{\mathcal{B}}(f)$.

\begin{Theorem}\label{BerkJuliaProp}
There is a unique extension of $f$ to a continuous self-map on $\mathbb{P}^{1,\mathcal{B}}_{v}$, which we again denote by $f$. The Berkovich Julia set $J_{v}^{\mathcal{B}}(f)$ is the smallest non-empty subset of $\mathbb{P}^{1,\mathcal{B}}_{v}$ such that
\begin{itemize}
\item[(i)] $J_{v}^{\mathcal{B}}(f)$ is completely $f$-invariant,
\item[(ii)] $J_{v}^{\mathcal{B}}(f)$ is compact,
\item[(iii)] $J_{v}^{\mathcal{B}}(f)$ contains no exceptional point, that is, no point with finite forward and finite backward orbits.
\end{itemize}
\end{Theorem}

\begin{proof} See \cite[\S 2.3, and Corollary 10.57]{BR}. \end{proof}

In complex dynamics, the Julia set is the closure of the repelling periodic points. In the classic non-archimedean setting this is an open conjecture (see \cite{Hs00}). The next proposition will be useful in the proof of Theorem \ref{totv}.

\begin{Proposition}\label{Juliaequal}
If the classical $v$-adic Julia set of $f$ is a non-empty subset of $\mathbb{P}^1(K_v)$, then
\[
J_v(f)=J_{v}^{\mathcal{B}}(f)
\]
and $J_v(f)$ is the closure of the repelling periodic points of $f$.
\end{Proposition}

\begin{proof}
Under our assumption, \eqref{Julia} tells us that $J_{v}^{\mathcal{B}}(f) \cap \mathbb{P}^1 (K_v) = J_v (f)$. Both sets on the left hand side are compact, and hence $J_v(f)$ is compact in the Hausdorff space $\mathbb{P}^{1,\mathcal{B}}_{v}$.

By assumption, $J_v(f)$ is not empty and by Theorem \ref{BerkJuliaProp}, it contains no exceptional point of $f$.
Since $J_v(f)$ is completely $f$-invariant in the classical sense, it is also completely $f$-invariant for the continuous extension of $f$ to a self-map on $\mathbb{P}^{1,\mathcal{B}}_{v}$ (see \cite[Proposition 2.15]{BR}).

We have just proved that $J_v(f)$ is a non-empty, compact and completely $f$-invariant subset of $\mathbb{P}^{1,\mathcal{B}}_{v}$ containing no exceptional point. By Theorem \ref{BerkJuliaProp}, we get $J_{v}^{\mathcal{B}}(f) \subseteq J_v (f)$, and hence the first claim of the proposition follows.

For the Berkovich Julia set $J_{v}^{\mathcal{B}}(f)$, Rivera-Letelier \cite[Theorem 10.88]{BR} proved that it is the Berkovich closure of the set of repelling periodic points of $f$ in $\mathbb{P}^{1,\mathcal{B}}_{v}$. But since $J_{v}^{\mathcal{B}}(f)$ equals $J_v (f)$, this is also true in the classical non-archimedean setting. This is precisely the second claim of the proposition.
\end{proof}

%%%%%%%%%%%%%
%%%%%%%%%%%%%

\section{Proof of Theorem 1.1} \label{proof}

\subsection{Preliminary lemmas} 
Let again $K$ be a number field with a non-archimedean valuation $v$.

\begin{Lemma}\label{Berklemma}
 The Berkovich closure of $\mathbb{P}^{1}(K)$ is $\mathbb{P}^1 (K_v)$. 
\end{Lemma}

\begin{proof} The space $\mathbb{P}^{1}(K_v)$ is a compact (and hence closed) subspace of $\mathbb{P}^{1,\mathcal{B}}_{v}$. The restriction of the Berkovich topology to $\mathbb{P}^{1}(K_v)$ is just the topology induced by the chordal metric $\rho_v$. Therefore, the closure of $\mathbb{P}^{1}(K)$ in the Berkovich topology is $\mathbb{P}^{1}(K_v)$ as claimed. \end{proof}

\begin{Lemma}\label{height}
 For any $f\in \overline{\mathbb{Q}}(x)$ of degree at least two and any $\sigma \in \Gal(\overline{\mathbb{Q}}/\mathbb{Q})$, we have $\widehat{h}_{f} = \widehat{h}_{\sigma(f)}\circ \sigma$.
\end{Lemma}

\begin{proof} This follows immediately from the defining properties \eqref{canonical height} of the canonical heights. \end{proof}

One of the main ingredients in the proof of Theorem \ref{totv} is the following result, which itself is an application of the celebrated equidistribution theorem of Yuan (see \cite[Theorem 3.7]{Yu08}).

\begin{Theorem}\label{v-adic Bogomolov}
Let $f \in \overline{\mathbb{Q}}(x)$ of degree at least two be such that $J^{\mathcal{B}}_{v} (f)$ is not contained in $\mathbb{P}^{1}(K_v)$. If $L/K$ is a Galois extension lying in $K_v$, then $L$ has property $(SB)$ relative to $\widehat{h}_f$.
\end{Theorem}

\begin{proof} See \cite[Theorem 2.5]{Po13} or \cite[Theorem 1]{FM12} for a proof. Note that in both cited theorems, the assumption is that $J^{\mathcal{B}}_{v} (f)$ shall not be contained in the Berkovich closure of $\mathbb{P}^{1}(K)$, which by Lemma \ref{Berklemma} is equivalent to the present formulation. \end{proof}

\subsection{Proof of Theorem \ref{totv}}
In addition to the previous lemmas we will use the substantial work of Rivera-Letelier, Baker and Rumely on dynamical systems on $\mathbb{P}^{1,\mathcal{B}}_{v}$, summarized in \cite[Chapter 10]{BR}.

\begin{proof}[Proof of Theorem \ref{totv}] We first prove that {\rm (i)} implies {\rm (ii)} as well as the Bogomolov property of $K^{tv}$ relative to $\widehat{h}_f$. This is equivalent to the statement that {\rm (i)} implies property $(SB)$ of $K^{tv}$ relative to $\widehat{h}_f$. So let $\sigma \in \Gal(\overline{\mathbb{Q}}/K)$ be such that $J_v (\sigma(f)) \not\subset \mathbb{P}^{1}(K_v)$ or $J_v(\sigma(f))=\emptyset$. Using equation \eqref{Julia} we infer that also $J^{\mathcal{B}}_{v} (\sigma(f)) \not\subseteq \mathbb{P}^{1}(K_v)$. Hence, by Theorem \ref{v-adic Bogomolov} we know that $K^{tv}$ has property $(SB)$ relative to $\widehat{h}_{\sigma(f)}=\widehat{h}_f \circ \sigma^{-1}$ (see Lemma \ref{height}). 
The extension $K^{tv}/K$ is Galois, and hence $\sigma^{-1}(K^{tv})=K^{tv}$. This proves the claim.

The implication ${\rm (ii)} \Rightarrow {\rm (iii)}$ is trivial, since $\PrePer(f)$ is an infinite set.

Next we will prove ${\rm (iii)} \Rightarrow {\rm (i)}$ by contraposition. So we assume that $J_v(\sigma(f))$ is a non-empty subset of $\mathbb{P}^1 (K_v)$ for all $\sigma \in \Gal(\overline{\mathbb{Q}}/K)$. Let $\sigma$ be in $\Gal(\overline{\mathbb{Q}}/K)$.
By Proposition \ref{Juliaequal} we have 
\begin{equation}\label{JuliaPropRef}
J_{v}^{\mathcal{B}}(\sigma(f))= J_v (\sigma(f)),
\end{equation}
 and $J_v (\sigma(f))$ is the closure of the set of repelling periodic points of $\sigma(f)$ in $\mathbb{P}^{1}(\mathbb{C}_v)$. Since any repelling periodic point of $\sigma(f)$ in $\mathbb{P}^{1,\mathcal{B}}_{v}$ belongs to $J_{v}^{\mathcal{B}}(\sigma(f))$, equation \eqref{JuliaPropRef} tells us that there is no repelling periodic point of $\sigma(f)$ in $\mathbb{P}^{1,\mathcal{B}}_{v}\setminus \mathbb{P}^{1}(\mathbb{C}_v)$. Hence we can apply another theorem of Rivera-Letelier \cite[Proposition 10.101]{BR} to deduce that there is at most one non-repelling periodic point of $\sigma(f)$ in $\mathbb{P}^{1}(\mathbb{C}_v)$. We conclude that 
\begin{equation}\label{onepoint}
\alpha \in J_v(\sigma(f)) \subseteq \mathbb{P}^1 (K_v) \text{ for all but possibly one } \alpha\in\Per(\sigma(f)).
\end{equation}
Since $J_v(\sigma(f))$ is completely $\sigma(f)$-invariant and contains infinitely many points from $K_v$, the rational map $\sigma(f)$ is necessarily defined over $K_v$. We write $\sigma(f)=g/h$ for polynomials $g,h\in K_v [x]$. For all $n\in\mathbb{N}_0$ there are polynomials $g_n,h_n \in K_v [x]$ such that $\sigma(f)^{(n)}=g_n/h_n$. The finite periodic points of $\sigma(f)$ are precisely the roots of the polynomials
\[
H_n (x)=g_n (x) - x h_n (x) \quad, \quad n \in \mathbb{N}_0.
\]
By \eqref{onepoint}, for each $n$ at most one root of $H_n$ does not lie in $K_v$. But since $H_n$ is also defined over $K_v$, all roots of $H_n$ must lie in $K_v$. This proves 
\begin{equation}\label{sigper}
\sigma(\Per(f))=\Per(\sigma(f))\subseteq \mathbb{P}^1 (K_v)
\end{equation}
for all $\sigma\in\Gal(\overline{\mathbb{Q}}/K)$, and hence $\Per(f)\subseteq \mathbb{P}^1 (K^{tv})$, as desired.

We will use \eqref{onepoint} to prove ${\rm (iv)} \Rightarrow {\rm (i)}$. We argue by contraposition, so let $f\in \overline{\mathbb{Q}}[x]$ of degree at least two be such that $J_v(\sigma(f))$ is a non-empty subset of $\mathbb{P}^1 (K_v)$ for all $\sigma \in \Gal(\overline{\mathbb{Q}}/K)$. In this case $\infty$ is a (super)attracting fixed point of $\sigma(f)$ for all $\sigma \in \Gal(\overline{\mathbb{Q}}/K)$, and we know from \eqref{onepoint} that it is the only non-repelling periodic point for these maps. The complete $\sigma(f)$-invariance of $J_v(\sigma(f))$ implies directly that 
\begin{equation}\label{sigpreper}
\sigma(\PrePer(f))\setminus[\infty]_{\sigma(f)} = \PrePer(\sigma(f))\setminus [\infty]_{\sigma(f)} \subseteq \mathbb{P}^1 (K_v).
\end{equation}
 Here $[\infty]_{\sigma(f)}$ denotes the union of all forward and backward orbits of $\infty$ under $\sigma(f)$. But since $\sigma(f)$ is a polynomial, we have $[\infty]_{\sigma(f)}=\{\infty\}$ for all $\sigma \in \Gal(\overline{\mathbb{Q}}/K)$. Hence, by \eqref{sigpreper} we have $\PrePer(f)\subseteq \mathbb{P}^1 (K^{tv})$, as desired.

The implication ${\rm (iii)} \Rightarrow {\rm (iv)}$ is trivial.
\end{proof}

\subsection{Examples}
We present two easy examples: one where the conditions of Theorem \ref{totv} are satisfied, and the other where they are not. Moreover, we will use Theorem \ref{totv} to reprove (an ineffective version of) a theorem due to Baker and Petsche \cite{BP05}.

\begin{Example}\label{goodreduction}
 Let $K^{tv}$ be as in Theorem \ref{totv} and $f \in \overline{\mathbb{Q}}(x)$ of degree at least two. If $f$ has good reduction at $v$, then $J_v (f)=\emptyset$ (see Lemma \ref{good reduction}). Hence, $K^{tv}$ has property $(SB)$ relative to $\widehat{h}_f$ whenever $f$ has good reduction at $v$. In particular, this is true for every $v$ if $f$ is a monic polynomial with algebraic integer coefficients. So, Theorem \ref{totv} implies the result of Bombieri and Zannier that $K^{tv}$ has the Bogomolov property relative to $h=\widehat{h}_{x^2}$.
 \end{Example}
 
 \begin{Example}\label{badreduction}
 Let $S$ be any finite set of rational primes, and denote by $k$ their product. The map $f(x)=(x^2 -x)/k$ has bad reduction at a prim $p$ if and only if $p\in S$. For $p\in S$ and $\alpha \in \mathbb{Q}_p$ with $\vert \alpha \vert_p \leq 1$, the preimages of $\alpha$ under $f$ are the solutions of the equation $x^2 -x -k\alpha=0$. Since $\vert k \alpha \vert_p \leq p^{-1}$, Hensel's Lemma yields $f^{-1}(\alpha) \subseteq \mathbb{Q}_p$. Moreover, both preimages of $\alpha$ have $p$-adic absolute value $\leq 1$.

Obviously $0$ is a fixed point of $f$. We claim that the backward orbit of $0$ under $f$ is contained in 
\[
F=\bigcap_{p\in S} \mathbb{Q}^{tp}.
\]
Let $\alpha_n \in \overline{\mathbb{Q}}$ be such that $f^{(n)}(\alpha_n)=0$. We know from the previous paragraph that $\alpha_n \in \mathbb{Q}_p$ for all $p\in S$. Since $f$ is defined over $\mathbb{Q}$, the same is true for every Galois conjugate of $\alpha_n$. We conclude that all the $\alpha_n$ are in $\mathbb{Q}^{tp}$ for all $p\in S$, proving the claim. In particular, there are infinitely many preperiodic points of $f$ in $F$. Since $f$ is a polynomial, Theorem \ref{totv} shows that all finite preperiodic points of $f$
are in $F$. Constructing such a backward orbit with $0$ replaced by any non-preperiodic element in $\mathbb{Z}$, we see that $F$ does not have property $(B)$ relative to $\widehat{h}_f$.
\end{Example}

\begin{Example}\label{lattesexample}
Let $E$ be an elliptic curve defined over the number field $K$. We may assume that $E$ is given in short Weierstra\ss{} form $y^2=x^3 + ax +b$, with $a,b\in K$.
The multiplication-by-$2$ map, denoted by $[2]$, defines a dynamical system on $E(\overline{\mathbb{Q}})$. If $\pi:E(\overline{\mathbb{Q}})\longrightarrow \mathbb{P}^1 (\overline{\mathbb{Q}})$ denotes the projection on the $x$-coordinate, there exists a rational map $f\in K(x)$ of degree $4$ such that 
\[
f(\pi(P))=\pi([2]P) \quad \text{for all } P\in E(\overline{\mathbb{Q}}).
\]
Such a map $f$ is called Latt\`es map. For details we refer the reader to \cite[Chapter 6]{Si07}. For any non-archimedean valuation $v$ on $K$, we have $J_v (f)=\emptyset$ \cite[Example 4.10]{Hs96}. Therefore, $K^{tv}$ has property $(SB)$ relative to $\widehat{h}_f$. It follows directly from the defining properties \eqref{canonical height} of the respective height functions that
\[
\widehat{h}_f \circ \pi = 2 \widehat{h}_E,
\]
where $\widehat{h}_E$ denotes the N\'eron-Tate height on $E$. Hence, $E$ has at most finitely many torsion points defined over $K^{tv}$, and there is a positive constant $c$ such that $\widehat{h}_E (P)\geq c$ for all $P\in E(K^{tv})$. Effective bounds on the number of torsion points and $c$ have been calculated in \cite[Section 6.3]{BP05}. 
\end{Example}

%%%%%%%%%%%%%
%%%%%%%%%%%%%

\section{On Narkiewicz's property $(P)$}\label{Narkiewicz}

\subsection{Field properties} In the previous section we have studied the properties $(B)$ and $(SB)$. We will start this section by recalling some more arithmetic properties on fields. The first definition is, just like the definition of the Bogomolov property, due to Bombieri and Zannier \cite{BZ01}.

\begin{Definition}\label{Northcott}
A field $F\subseteq \overline{\mathbb{Q}}$ has the \emph{Northcott property} $(N)$ if any subset of $F$ of bounded height is finite.
\end{Definition}

Northcott \cite{No50} proved that there are only finitely many points of bounded height and bounded degree in $\overline{\mathbb{Q}}$. Hence, every number field satisfies property $(N)$.
Note that there is no reason to define a Northcott property relative to some canonical height $\widehat{h}_f$. Two height functions differ only by a bounded function, and hence any set of bounded canonical height $\widehat{h}_f$ is a set of bounded height $h$ and vice versa.

Narkiewicz introduced the following definition.

\begin{Definition}\label{PropertyP}
A field $F$ has \emph{property $(P)$} if every polynomial $f(x)\in F[x]\setminus F$ such that there is an infinite set $X\subseteq F$ with $f(X)=X$, is of degree one. If a field $F$ satisfies the same statement with ``polynomial" replaced by `rational map", then we say that $F$ has \emph{property $(R)$}. 
\end{Definition}

The next definition of Liardet goes in the same direction. We denote by $\overline{F}$ a fixed algebraic closure of the field $F$.

\begin{Definition}\label{PropertyPbar} 
A field $F$ has \emph{property $(\overline{P})$} if every polynomial $f(x)\in \overline{F}[x]\setminus F$ such that there is an infinite set $X\subseteq \overline{F}$ consisting of elements of uniformly bounded degree over $F$, with $X\subseteq f(X)$, is of degree one. Again the obvious generalization for rational functions is called \emph{property $(\overline{R})$}. 
\end{Definition}

For a subfield $F$ of $\overline{\mathbb{Q}}$ we have the following implications:
\begin{align}\label{implications} 
(N)\Longrightarrow (\overline{R}) \Longrightarrow (R) \Longrightarrow (P).
\end{align}
We refer the reader to \cite[Section 6]{CW13} for a proof of this result and a more detailed discussion of these properties (see also \cite{DZ08}). The authors of \cite{CW13} only state the result $(N)\Longrightarrow (\overline{P})$ \cite[Lemma 6.7.]{CW13}, but their proof is valid for $(N)\Longrightarrow (\overline{R})$ as well. 

It is still open whether $(P)\Longleftrightarrow (R)$. Dvornicich and Zannier \cite[Theorem 3]{DZ07} gave an example of a field where property $(P)$ is not preserved under a finite extension. Since the properties $(N)$ (\cite[Theorem 2.1]{DZ08}), $(\overline{P})$ and $(\overline{R})$ (obvious) are stable under finite extensions, none of them is equivalent to $(P)$. 

For a field $F$ and any positive integer $d$ we denote by $F^{(d)}$ the compositum in $\overline{F}$ of all field extensions of $F$ of degree at most $d$. In \cite[Probl\`eme 415]{Na63} Narkiewicz (implicitly) conjectured the following (see also \cite[Problem 10]{Na71}). 

\begin{Conjecture}[Narkiewicz (1963)] \label{conj}
The field $\mathbb{Q}^{(d)}$ has property $(P)$ for all positive integers $d$. 
\end{Conjecture}

Bombieri and Zannier \cite[Corollary 1]{BZ01} proved that $\mathbb{Q}^{(2)}$ has property $(N)$, and by \eqref{implications} also properties $(\overline{R})$ and $(P)$. Narkiewicz conjectured further that the statement in \ref{conj} is also true if we replace $\mathbb{Q}$ by any completely transcendental extension field. We will only focus on the case of algebraic numbers, although some of the tools used in the proof may apply in this general setting as well.

\subsection{Proof of Conjecture \ref{conj}}
In what follows we will give a positive answer to Narkiewicz's conjecture. To ease notation we will define yet another property, introducing dynamical canonical heights from \eqref{canonical height}.

\begin{Definition}\label{superstrongBogomolov} 
A field $F\subseteq \overline{\mathbb{Q}}$ has the \emph{universal strong Bogomolov property $(USB)$} if $F$ has the strong Bogomolov property relative to $\widehat{h}_f$ for all $f\in\overline{\mathbb{Q}}(x)$ with $\deg(f)\geq2$.
\end{Definition}

Of course, property $(N)$ implies property $(USB)$. The next easy lemma relates property $(USB)$ to the other properties defined above.

\begin{Lemma}\label{LemmaSSB}
Let $F\subseteq\overline{\mathbb{Q}}$ be a field with property $(USB)$. Then $F$ has property $(R)$. If also $F^{(d)}$ has property $(USB)$ for all positive integers $d$, then $F$ has property $(\overline{R})$.
\end{Lemma}

\begin{proof} Let $F\subseteq \overline{\mathbb{Q}}$ be a field such that $F^{(d)}$ has property $(USB)$ for some $d\in \mathbb{N}$. Moreover, let $f\in\overline{\mathbb{Q}}(x)$ be of degree at least two, and $X \subseteq \overline{\mathbb{Q}}$ a non-empty set of points of degree at most $d$ over $F$ with $X\subseteq f(X)$. Assume there is a non-preperiodic point $\alpha_0 \in X$. Then $X\subseteq f(X)$ implies the existence of a sequence $\alpha_1,\alpha_2,\dots $ of pairwise distinct elements in $X$ such that $f(\alpha_i)=\alpha_{i-1}$ for all $i \in \mathbb{N}$. In particular, we get 
\[
0\neq \widehat{h}_f (\alpha_n)= \frac{1}{\deg(f)^n} \widehat{h}_f (\alpha_0) \longrightarrow 0 \quad \text{as } n\longrightarrow \infty.
\]
Hence, $F^{(d)}$ contains points of arbitrarily small canonical height $\widehat{h}_f$, which contradicts our hypothesis that $F^{(d)}$ has property $(USB)$. Therefore, all points of $X$ are preperiodic for $f$. But property $(USB)$ of $F^{(d)}$ also implies that there are only finitely many preperiodic points of $f$ in $X\subseteq F^{(d)}$, and hence $X$ is finite.
The case $d=1$ proves the first statement of the lemma. \end{proof}

\begin{Theorem}\label{tada}
Let $K$ be a number field and $S$ be an infinite set of non-archimedean valuations on $K$. If $L/K$ is a Galois extension such that $d_{w\mid v}=[L_w:K_v]$ is finite for all $v \in S$ and $w\mid v$ on $L$, then $L$ has property $(USB)$.
\end{Theorem}

\begin{proof} Let $L/K$ be as above and let $f\in \overline{\mathbb{Q}}(x)$ be of degree at least $2$. By Lemma \ref{good reduction}, $f$ has bad reduction only above finitely many primes. Therefore we can fix a $v\in S$ where $f$ has good reduction.

Since the extension $L/K$ is Galois, the local degree $d_v=[L_w : K_v]$ does not depend on the choice of $w\mid v$. We fix such a $w$. Since $d_v$ is finite, a standard application of Krasner's Lemma yields the existence of an $\alpha \in \overline{\mathbb{Q}}$ sucht that $L_w = K_v (\alpha)$.
Let $w'$ be a valuation on $L(\alpha)$ with $w' \vert_{L} = w$, and set $v' = w' \vert_{K(\alpha)}$. Then $L(\alpha)/K(\alpha)$ is Galois and 
\[
L(\alpha)\subseteq L(\alpha)_{w'} = L_w =K_v (\alpha) = K(\alpha)_{v'}.
\]
The map $f$ has good reduction at $v'\mid v$, implying that $J_{v'}(f)=\emptyset$ (see Lemma \ref{good reduction}). Thus, by Theorem \ref{totv}, the field $L(\alpha) \supseteq L$ has property $(SB)$ relative to $\widehat{h}_f$.
As the rational function $f$ was chosen arbitrarily, $L$ has property $(USB)$.
\end{proof}

\begin{Remark}
We have seen in Example \ref{badreduction} $b)$ that the conclusion of the theorem is false for any finite set $S$ of non-archimedean absolute values.
\end{Remark}

\begin{Lemma}\label{Lemmatada}
Let $L/K$ be as in Theorem \ref{tada} and $d$ any positive integer. Then also $L^{(d)}$ is of this type, i.e. a Galois extension of $K$ with finite local degrees at all places $v$ in the infinite set $S$.
\end{Lemma}

\begin{proof} The proof that $L^{(d)}/K$ is Galois is easy and proceeds as in \cite[Theorem 2.1]{CW13}. The finiteness of local degrees is also well known and follows as in \cite[Proposition 1]{BZ01} from the fact that for all non-archimedean places $v$ on $K$ there are only finitely many extensions of $K_v$ of fixed degree. \end{proof}

The following immediate consequence proves Conjecture \ref{conj}.

\begin{Corollary}\label{CorNar}
The field $\mathbb{Q}^{(d)}$ has properties $(USB)$ and $(\overline{R})$.
\end{Corollary}

\begin{proof} This follows from Lemma \ref{Lemmatada} and Theorem \ref{tada}, if we start with the field $L=\mathbb{Q}$ and the set of all primes $S$.
\end{proof}

\begin{Example}\label{RCF}
We will present a new class of examples satisfying property $(\overline{R})$, using Theorem \ref{tada}. Let $K$ be an imaginary quadratic number field and $\mathcal{O}_K$ the ring of integers of $K$. For simplicity we assume that no non-trivial roots of unity are contained in $K$. We set
\[
S'=\{p \text{ prime }\vert p \text{ is inert in } \mathcal{O}_K \}. 
\] 
Let $S\subseteq S'$ be a non-empty subset of $S'$. For all $p\in S'\setminus S$ and all $n\in\mathbb{N}$ we denote by $K_{p^n}$ the ring class field associated to the order $\mathbb{Z}+p^n \mathcal{O}_K$. We have $K_{p^n}\subseteq K_{p^{n+1}}$ for all $n\in\mathbb{N}$, and hence 
\[
K[p]= \bigcup_{n\in\mathbb{N}}K_{p^n}
\]
is a field, and in particular a Galois extension of $K$. Let $\ell\in S$. By definition $\ell \mathcal{O}_K$ is a prime ideal. It follows from the construction of ring class fields that this prime ideal splits completely in $K[p]$. This means that
\[
K[p] \subseteq K^{t\ell} \text{ for all } p \in S' \setminus S \text{ and all } \ell \in S. 
\]
Define $F_S$ as the compositum of all fields $K[p]$, $p\in S'\setminus S$. Then $F_S$ is a Galois extension of $K$ with $F_S \subseteq K^{t\ell} \text { for all } \ell \in S.$
We apply Theorem \ref{tada} to see that $F_S$ has property $(USB)$ whenever $S$ is infinite. Moreover, Lemma \ref{Lemmatada} also implies that such a field $F_S$ has property $(\overline{R})$.

A classical result of Hasse \cite{Ha31} tells us that for almost all $p\in S' \setminus S$, the ramification index of $K_{p^n}/K$ at $p$ is at least $p^n$.  Therefore, $F_S / K$ has infinite local degree at almost all primes $p \in S' \setminus S$.
\end{Example}

\subsection{Open Problems}
Let $p_1,p_2,\dots $ be distinct primes, and $d_1,d_2,\dots $ natural numbers such that $\vert p_i ^{1/d_i} \vert $ tends to infinity as $i\rightarrow \infty$. Widmer \cite[Corollary 2]{Wi11} proved that the field
\[
F=\mathbb{Q}(p_1^{1/d_1}, p_2^{1/d_i},\dots)
\] 
has property $(N)$, and hence property $(USB)$. This field is either contained in some $\mathbb{Q}^{(d)}$ (if the $d_i$ are uniformly bounded), or has infinite local degree at all primes. So, Theorem \ref{tada} cannot give a full classification of all fields with property $(USB)$. Notice that in the case of unbounded numbers $d_i$, the Galois closure of F over $\mathbb{Q}$ contains infinitely many roots of unity. In that case $F$ has none of the properties $(N)$, $(P)$ or $(USB)$. 

\begin{question}\label{question1}
Do all fields $F\subseteq\overline{\mathbb{Q}}$ with property $(USB)$ which are Galois extensions of some number field $K$ have finite local degree at infinitely many primes of $K$?
\end{question}

As Dvornicich and Zannier \cite[Theorem 3]{DZ07} have constructed a Galois extension of $\mathbb{Q}$ with infinite local degree at all primes and satisfying property $(P)$ but not $(N)$, Question \ref{question1} is strongly related to the next one.

\begin{question}\label{question2}
Is property $(USB)$ equivalent to $(N)$ or $(P)$?
\end{question}

By Corollary \ref{CorNar} a proof of the equivalence of $(N)$ and $(USB)$ would answer Bombieri's and Zannier's question whether $\mathbb{Q}^{(d)}$ has property $(N)$ for all integers $d$.  At the moment it seems to be out of reach to decide whether all fields with bounded local degrees above all primes have property $(N)$. In order to give an answer to Question \ref{question2}, but also independently of it, we ask:

\begin{question}
Let $F_S$ be as in Example \ref{RCF} with $S$ an infinite set of primes. Does $F_S$ satisfy property $(N)$?
\end{question}

Note that $F_S$ has property $(USB)$ but (in the notation of Example \ref{RCF}) there are infinitely many primes of $K$ where $F_S / K$ has infinite local degree, whenever $S$ and $S' \setminus S$ are infinite sets.

\subsection*{Acknowledgments} This work was supported by the DFG-Projekt `Heights and unlikely intersections" HA~6828/1-1. I am grateful to Paul Fili for helpful discussions and for answering a question on Berkovich spaces. Moreover, I would like to thank Liang-Chung Hsia for providing the reference \cite{Hs96}, and Philipp Habegger for pointing out the relevant properties of ring class fields, leading to Example \ref{RCF}.


\begin{thebibliography}{[99]}

\bibitem{AZ00}{\textsc{F. Amoroso} and \textsc{U. Zannier}, {\it A Relative Dobrowolski Lower Bound over Abelian Extensions}, Ann. Scuola Norm. Sup. Pisa Cl. Sci. (4) 29 (2000), 711--727}

\bibitem{BP05}{\textsc{M. Baker} and \textsc{C. Petsche}, {\it Global discrepancy and small points on elliptic curves}, Int. Math. Res. Not. 61 (2005), 3791--3834}

\bibitem{BR06}{\textsc{M. Baker} and \textsc{R. Rumely}, {\it Equidistribution of small points, rational dynamics, and potential theory}, Ann. Inst. Fourier (Grenoble) 56 (2006), 625--688}

\bibitem{BR}{\textsc{M. Baker} and \textsc{R. Rumely}, {``Potential Theory and Dynamics on the Berkovich Projective Line''}, Math. Surveys and Monogr. 159, Amer. Math. Soc., Providence, RI, 2010}

\bibitem{Ber}{\textsc{V. Berkovich}, {``Spectral Theory and Analytic Geometry over Non-Archimedean Fields''}, Math. Surveys Monogr. 33, Amer. Math. Soc., Providence, RI, 1990}

\bibitem{BG}{\textsc{E. Bombieri} and \textsc{W. Gubler}, {``Heights in Diophantine Geometry''}, New Math. Monogr. 4, Cambridge Univ. Press, Cambridge, 2006}

\bibitem{BZ01}{\textsc{E. Bombieri} and \textsc{U. Zannier}, { \it A note on heights in certain infinite extensions of $\mathbb{Q}$}, Rend. Lincei Mat. Appl. 12 (2001), 5--14}

\bibitem{CS93}{\textsc{G. S. Call} and \textsc{J. H. Silverman}, {\it Canonical heights on varieties with morphisms}, Compos. Math. 89 (1993), 163--205}

\bibitem{CL06}{\textsc{A. Chambert-Loir}, {\it Mesures et \'equidistribution sur les espaces de Berkovich}, J. Reine Angew. Math. 595 (2006), 215--235}

\bibitem{CW13}{\textsc{S. Checcoli} and \textsc{M. Widmer}, {\it On the Northcott property and other properties related to polynomial mappings}, Math. Proc. Cambridge Philos. Soc. 155 (2013), 1--12}

\bibitem{DZ07}{\textsc{R. Dvornicich} and \textsc{U. Zannier}, {\it Cyclotomic Diophantine problems (Hilbert irreducibility and invariant sets for polynomial maps)}, Duke Math. J. 139 (2007), 527--554}

\bibitem{DZ08}{\textsc{R. Dvornicich} and \textsc{U. Zannier}, {\it On the properties of Northcott and Narkiewicz for fields of algebraic numbers}, Funct. Approx. Comment. Math. 39 (2008), 163--173}

\bibitem{FRL04}{\textsc{C. Favre} and \textsc{J. Rivera-Letelier}, {\it Th\'eor\`eme d'equidistribution de Brolin en dynamique $p$-adique},  C. R. Math. Acad. Sci. Paris 339 (2004), 271--276}

\bibitem{FM12}{\textsc{P. Fili} and \textsc{Z. Miner}, {\it Equidistribution and the heights of totally real and totally $p$-adic numbers}, Acta Arith. 170 (2015), 15--25}

\bibitem{Ha13}{\textsc{P. Habegger}, {\it Small height and infinite non-abelian extensions}, Duke Math. J. 162 (2013), 2027--2076}

\bibitem{Ha31}{\textsc{H. Hasse}, {\it Das Zerlegungsgesetz f\"ur die Teiler des Moduls in den Ringklassenk\"orpern der komplexen Multiplikation},  Monatsh. Math. Phys. 38 (1931), 331--344}

\bibitem{Hs96}{\textsc{L.-C. Hsia}, { \it A weak N\'eron model with applications to $p$-adic dynamical systems},  Compos. Math. 100 (1996), 277--304}

\bibitem{Hs00}{\textsc{L.-C. Hsia}, { \it Closure of periodic points over a non-Archimedean field}, J. London Math. Soc. (2) 62 (2000), 685--700}

\bibitem{KL76}{\textsc{K. K. Kubota} and \textsc{P. Liardet}, {\it R\'efutation d'une conjecture de W. Narkiewicz}, C. R. Acad. Sci. Paris S\'er. A-B 282 (1976), A1261-A1264}

\bibitem{Na63}{\textsc{W. Narkiewicz}, { \it Probl\`eme 414}, Colloq. Math.  10 (1963), 186--187}

\bibitem{Na71}{\textsc{W. Narkiewicz}, { \it Some unsolved problems}, in: Colloque de Th\'eorie des Nombres (Bordeaux, 1969),  Bull. Soc. Math. France, Mem. 25 (1971),159--164}

\bibitem{No50}{Northcott, D. G.: {\it Periodic points on an algebraic variety}, Ann. of Math. 51 (1950), 167--177}

\bibitem{Po13}{\textsc{L. Pottmeyer}, { \it Heights and totally real numbers}, Rend. Lincei Mat. Appl. 24 (2013), 471--483}

\bibitem{RiL03}{\textsc{J. Rivera-Letelier}, {\it Espace hyperbolique $p$-adique et dynamique des fonctions rationnelles},  Compos. Math. 138 (2003), 199--231}

\bibitem{Sch73}{\textsc{A. Schinzel}, {\it On the product of the conjugates outside the unit circle of an algebraic number}, Acta Arith. 24 (1973), 385--399}

\bibitem{Si07}{\textsc{J. H. Silverman}, {``The Arithmetic of Dynamical Systems''}, Grad. Texts in Math. 241, Springer, New York, 2007}

\bibitem{Wi11}{\textsc{M. Widmer}, {\it On certain infinite extensions of the rationals with Northcott property}, Monatsh. Math. 162 (2011), 341--353}

\bibitem{Yu08}{\textsc{X. Yuan}, {\it Big line bundles over arithmetic varieties}, Invent. Math. 173 (2008), 603--649}

\end{thebibliography}
\end{document}